\documentclass[a4paper,12pt]{article}

\newenvironment{proof}{\noindent {\bf Proof:}}{\hfill $\Box$}

\newtheorem{theorem}{Theorem}
\newtheorem{lemma}{Lemma}

\newtheorem{corollary}{Corollary}

\newtheorem{definition}{Definition}
\newtheorem{assumption}{Assumption}
\newtheorem{remark}{Remark}

\usepackage{booktabs}

\usepackage{algorithm}
\usepackage[noend]{algpseudocode}
\makeatletter
\def\BState{\State\hskip-\ALG@thistlm}
\makeatother
\usepackage{algorithmicx}

\usepackage{amsmath}
\usepackage{amssymb}
\usepackage{amsfonts}
\usepackage{graphicx}
\usepackage{ dsfont }

\usepackage[normalem]{ulem}
\usepackage{color}

\usepackage{tikz}
\usepackage{lipsum}

\newcommand{\mr}[1]{\mathrm{#1}}

\newcommand{\Mc}{\mathcal{M}}

\newcommand{\Pc}{\mathcal{P}}
\newcommand{\Rb}{\mathbb{R}}
\newcommand{\Cb}{\mathbb{C}}
\newcommand{\Nb}{\mathbb{N}}

\newcommand{\bs}{\boldsymbol}

\newcommand{\Fc}{\mathcal{F}}

\newcommand{\Ac}{\mathcal{A}}
\newcommand{\Kc}{\mathcal{K}}
\newcommand{\tow}{\xrightarrow{w}}
\DeclareMathOperator*{\Argmin}{Arg\,min}

\usepackage{url}

\newcommand*{\herm}{^{\mathsf{H}}}

\renewcommand{\Re}{\operatorname{Re}}
\renewcommand{\Im}{\operatorname{Im}}

\newcommand{\argmin}{\operatornamewithlimits{arg\,min}}

\newcommand{\new}[1]{{\color{black}#1}}
\def\be{\begin{equation}}
\def\ee{\end{equation}}

\title{\bf On Convergence of Extended Dynamic Mode Decomposition to the Koopman Operator}

\begin{document}

\author{Milan Korda$^1$, Igor Mezi{\'c}$^1$}

\footnotetext[1]{Milan Korda and Igor Mezi{\'c} are with the University of California, Santa Barbara,\; {\tt milan.korda@engineering.ucsb.edu, mezic@engineering.ucsb.edu}}

\maketitle
\begin{abstract}
\looseness-1 Extended Dynamic Mode Decomposition (EDMD) \cite{Williamsetal:2015} is an algorithm that approximates the action of the Koopman operator on an $N$-dimensional subspace of the space of observables by sampling at $M$ points in the state space. Assuming that the samples are drawn either independently or ergodically from some measure $\mu$, it was shown in~\cite{klus2015numerical} that, in the limit as $M\rightarrow\infty$, the EDMD operator $\mathcal{K}_{N,M}$ converges to $\mathcal{K}_N$, where $\mathcal{K}_N$ is the $L_2(\mu)$-orthogonal projection of the action of the Koopman operator on the finite-dimensional subspace of observables. We show that, as $N \rightarrow \infty$, the operator $\mathcal{K}_N$ converges in the strong operator topology to the Koopman operator. This in particular implies convergence of the predictions of future values of a given observable over any finite time horizon, a fact important for practical applications such as forecasting, estimation and control. In addition, we show that accumulation points of the spectra of $\mathcal{K}_N$ correspond to the eigenvalues of the Koopman operator with the associated eigenfunctions converging weakly to an eigenfunction of the Koopman operator, provided that the weak limit of the eigenfunctions is nonzero. As a by-product, we propose an analytic version of the EDMD algorithm which, under some assumptions, allows one to construct  $\mathcal{K}_N$ directly, without the use of sampling. Finally, under additional assumptions, we analyze convergence of $\mathcal{K}_{N,N}$ (i.e., $M=N$), proving convergence, along a subsequence, to weak eigenfunctions (or eigendistributions) related to the eigenmeasures of the Perron-Frobenius operator. No assumptions on the observables belonging to a finite-dimensional invariant subspace of the Koopman operator are required throughout.
 \end{abstract}

\begin{center}\small
\end{center} 
 
\section{Introduction}
Recently, there has been an expanding interest in utilizing  the {\it spectral expansion} based methodology that enabled progress in data-driven analysis of high-dimensional nonlinear systems. The research was  initiated  in \cite{MezicandBanaszuk:2004,Mezic:2005}, using composition (Koopman) operator representation  originally defined in~\cite{Koopman:1931}. The framework is being used for  model reduction, identification, prediction, data assimilation and control of deterministic (e.g., \cite{Budisicetal:2012,mauroy2017koopman,rowley2009,Williamsetal:2015,Bruntonetal:2016,Giannakisetal:2015,korda2016linear}) as well as stochastic dynamical systems~(e.g.,~\cite{takeishi2017subspace,riseth2017operator,wu2017variational}). This has propelled the theory to wide use on a diverse set of applications  such as fluid dynamics \cite{SharmaMezicM16,Glazetal:2016}, power grid dynamics \cite{Susukietal:2016}, neurodynamics \cite{brunton2016extracting}, energy efficiency \cite{GeorgescuandMezic:2015},  molecular dynamics \cite{wuetal:2016} and data fusion \cite{Williamsetal:2015}.

Numerical  methods for approximation of the spectral properties of the Koopman operator have been considered since the inception of the  data-driven analysis of dissipative dynamical systems  \cite{MezicandBanaszuk:2004,Mezic:2005}. These belong to the class of Generalized Laplace Analysis (GLA) methods \cite{Mezic:2013}. An alternative line of algorithms, called the Dynamic Mode Decomposition (DMD) algorithm \cite{schmid:2010,rowley2009} have also been advanced, enabling concurrent data-driven determination of approximate eigenvalues and eigenvectors of the underlying DMD operator. 
The examples of  DMD-type algorithms are 1) the companion-matrix method proposed by Rowley et al. \cite{rowley2009}, 2) the SVD-based DMD developed by Schmid \cite{schmid:2010},  3) the Exact DMD method introduced by Tu et al. \cite{tu2014dynamic}, 4) the Extended DMD~\cite{Williamsetal:2015}.
The relationship between these methods and the spectral operator properties of the Koopman operator was first noticed in \cite{rowley2009},  based on the spectral expansion developed in \cite{Mezic:2005}.  However, rigorous results in this direction are sparse. Williams et al. \cite{Williamsetal:2015} provided a result the corollary of which is that the spectrum of the EDMD approximation is contained in the spectrum of the Koopman operator provided the observables belong to a finite-dimensional invariant subspace of the Koopman operator and the data matrix is of the same rank. The work of Arbabi and Mezi\'c \cite{ArbabiandMezic:2016} suggested that an alternative assumption to the finite rank is that the number of sampling points $M$ goes to infinity even though the results of~\cite{ArbabiandMezic:2016} still implicitly rely on a finite-dimensional assumption. 


The work \cite{klus2015numerical} showed that, assuming either independent identically distributed (iid) or ergodic sampling from a measure $\mu$, the EDMD operator on $N$ observables constructed using $M$ samples, $\Kc_{N,M}$, converges as $M\to\infty$ to $\Kc_N$, where $\Kc_N$ is the $L_2(\mu)$-orthogonal projection of the action of the Koopman operator on the finite-dimensional subspace of observables. In this work, we show that $\Kc_N$ converges to $\Kc$ in the strong operator topology. As a result, predictions of a given observable obtained using $\Kc_N$ or $\Kc_{N,M}$ over any finite prediction horizon converge in the $L_2(\mu)$ norm (as $N$ or $N$ and $M$ tend to infinity) to its true values. In addition, we show that, as $N\rightarrow \infty$, accumulation points of the spectra of $\Kc_N$ correspond to eigenvalues of the Koopman operator, and the associated eigenfunctions converge weakly to an eigenfunction of the Koopman operator, provided that the weak limit of the eigenfunctions is nonzero. The results hold in a very general setting with minimal underlying assumptions. In particular, we {\it do not} not assume that the finite-dimensional subspace of observables is invariant under the action of the Koopman operator or that the dynamics is measure preserving.

As a by-product of our results, we propose an analytic version of the EDMD algorithm which allows one to construct  $\Kc_N$ directly, without the need for sampling, under the assumption that the transition mapping of the dynamical system is known analytically and the $N$-dimensional subspace of observables is such that the integrals of the products of the observables precomposed with the transition mapping can be evaluated in closed form. This method is not immediately useful for large-scale data-driven applications that the EDMD was originally designed for but it may prove useful in control applications (e.g.,~\cite{korda2016linear}), where model is often known, or for numerical studies of Koopman operator approximations on classical examples, eliminating the sampling error in both cases.

Finally, we analyze convergence of $\Kc_{N,N}$, i.e., the situation where the number of samples $M$ and the number of observables $N$ are equal. Under the additional assumptions that the sample points lie on the same trajectory and the mapping $T$ is a homeomorphism, we prove convergence along a subsequence to weak eigenfunctions (or eigendistributions in the sense of~\cite{GelfandandShilov:1964}) of the Koopman operator, which also turn out to be eigenmeasures of the Perron-Frobenius operator.

The paper is organized as follows: in Section~\ref{EDMD} we introduce the setting of EDMD. In Section~\ref{EDMDl2} we show that EDMD is an orthogonal projection of the action of the Koopman operator on a finite subspace of observables with respect to the empirical measure supported on sample points drawn from a measure $\mu$. In Section~\ref{EDMDconv} we show that this projection converges to the $L_2(\mu)$-projection of the action of the Koopman operator. In Section \ref{KNconv} we analyze the convergence of the EDMD approximations as the dimension of the subspace $N$ goes to infinity, showing convergence in strong operator topology and convergence of the eigenvalues along a subsequence plus weak convergence of the associated eigenfunctions. In Section~\ref{sec:finiteHor} we show convergence of finite-horizon predictions. Section~\ref{analyEDMD} describes the analytic construction of~$\Kc_N$. Section~\ref{KNNconv} contains results for the case when $M=N$ and only convergence to weak eigenfunctions, along a subsequence, is  proven. We conclude in Section~\ref{conc}.

\new{
\subsection*{Notation}
The spaces of real and complex numbers are denoted by $\Rb$ and $\Cb$ with $\Rb^{n\times k}$ and $\Cb^{n\times k}$ denoting the corresponding real and complex $n\times k$ matrices; the real and complex column vectors are denoted by $\Rb^n := \Rb^{n\times 1}$ and  $\Cb^n := \Cb^{n\times 1}$. The complex conjugate of $a\in \Cb$ is denoted by $\overline{a}$. Given a matrix $A\in \Cb^{n\times k}$, $A^\top$ denotes its transpose and $A\herm$ denotes its Hermitian transpose (i.e., $A_{i,j}^\top = A_{j,i}$ and $A\herm_{i,j} = \overline{A_{j,i}}$). The Moore-Penrose pseudoinverse of a matrix $A$ is denoted by~$A^\dagger$. The Frobenius norm of a matrix $A$ is denoted by $\| A \|_F = \sqrt{ \sum_{i,j}{ |A_{i,j}|^2 }}$. Given a vector $c\in \Cb^n$, the symbol $\|c\|_2 := \sqrt{\sum_{i}|c_i|^2}$ denotes its Euclidean norm.
}

\section{Extended Dynamic Mode Decomposition}
\label{EDMD}
We consider a discrete time dynamical system
\begin{equation}\label{eq:sys}
x^+ = T(x)
\end{equation}
with $T:\Mc \to \Mc$, where $\Mc$ is a topological space\footnote{We choose to work in the general setting of dynamical systems on arbitrary topological spaces which encompasses dynamical systems on finite-dimensional manifolds (in which case one can regard $\Mc$ as a subset of $\Rb^n$), as well as infinite dimensional dynamical systems, arising, for example, from the study of partial differential equations or dynamical systems with control inputs~\cite{korda2016linear}.}, and we assume that we are given snapshots of data
\begin{equation}\label{eq:data}
\bs  X = [x_1,\ldots, x_M],\quad \bs Y = [y_1,\ldots, y_M]
\end{equation}
satisfying $y_i = T(x_i)$. We do not assume that the data points line on a single trajectory of~(\ref{eq:sys}).

Given a vector space of observables $\mathcal{F}$ such that $\psi: \Mc \to \Cb$ and $\psi \circ T \in \Fc$ for every $\psi \in \Fc$, we define the Koopman $\Kc: \Fc\to \Fc$ by
\[
\Kc \psi = \psi \circ T,
\]
where $\circ$ denotes the pointwise function composition. Given a set of linearly independent basis functions $\psi_i \in \Fc$, $i=1,\ldots,N$, and defining
\begin{equation}\label{eq:FNdef}
\Fc_N := \mr{span}\{\psi_1,\ldots,\psi_N\},
\end{equation}
the EDMD constructs a finite-dimensional approximation $\Kc_{N,M}:\Fc_N\to \Fc_N$ of the Koopman operator by solving the least-squares problem
\begin{equation}\label{eq:ls}
\min_{A\in \Cb^{N\times N}} \| A\bs \psi(\bs X) - \bs\psi(\bs  Y)  \|_F^2 = \min_{A\in \Cb^{N\times N}} \sum_{i=1}^M \| A\bs \psi(x_i) - \bs\psi(y_i)  \|_2^2,
\end{equation}
where
\[
\bs\psi(\bs  X) = [\bs \psi(x_1),\ldots, \bs \psi(x_M)],\quad \bs\psi(\bs  Y) = [\bs \psi(y_1),\ldots, \bs \psi(y_M)]
\]
and 
\[
\bs \psi(x) = [\psi_1(x),\ldots,\psi_N(x)]^\top.
\]
Denoting \begin{equation}\label{eq:ANM}
A_{N,M} = \bs\psi(\bs Y)\bs\psi(\bs X)^\dagger,
\end{equation}
a solution\footnote{In general, the solution to~(\ref{eq:ls}) may not be unique; however, the matrix $A_{N,M} = \bs\psi(\bs Y)\bs\psi(\bs X)^\dagger$, where $\cdot^
\dagger$ denotes the Moore-Penrose pseudoinverse, is always uniquely defined and $A_{N,M}$ is always a minimizer in~(\ref{eq:ls}).  } to~(\ref{eq:ls}), the finite-dimensional approximation of the Koopman operator 
\[
\Kc_{N,M} : \Fc_N \to \Fc_N
\]
is then defined by
\begin{equation}\label{eq:KNM_def}
\Kc_{N,M} \phi = c_\phi\herm A_{N,M}\bs\psi 
\end{equation}
for any $\phi = c_\phi\herm \bs\psi$, $c_\phi\in\Cb^N$ (i.e., for any $\phi \in \Fc_N$). The operator $\Kc_{N,M}$ will be referred to as the EDMD operator.

\section{EDMD as $L_2$ projection}
\label{EDMDl2}
To the best of our knowledge, the results of this section and Section~\ref{EDMDconv} were first obtained in~\cite[Section~3.4]{klus2015numerical} and hinted at already in~\cite{Williamsetal:2015}. Here we rephrase these results in a form more suitable for our purposes.

\new{

From here on we assume\footnote{\new{Since $\Kc:\Fc\to\Fc$, the assumption of $\Fc = L_2(\mu)$ implies that the composition relation $\phi\circ T$, $\phi\in L_2(\mu)$, gives rise to a well-defined operator from $L_2(\mu)$ to $L_2(\mu)$. In particular this implies that $\| \phi_1\circ T - \phi_2\circ T\|_{L_2(\mu)} = 0$ whenever $\| \phi_1-\phi_2 \|_{L_2(\mu)} = 0$ and that $\int_{\Mc} |\phi\circ T|^2 \, d\mu < \infty$ for all $\phi \in L_2(\mu)$.}} that $\Fc = L_2(\mu)$, where $\mu$ is a given positive measure on $\Mc$. This assumption in particular implies that the basis functions $\psi_i$ in~(\ref{eq:FNdef}) belong to $L_2(\mu)$ and hence $\Fc_N$ is a closed (but not necessarily invariant) subspace of $L_2(\mu)$. Note that the measure $\mu$ is \emph{not} required to be invariant for~(\ref{eq:sys}) and hence the Koopman operator $\Kc$ is not necessarily unitary. In practical applications, the measure $\mu$ will typically be the uniform measure on $\Mc$ or other measure from which samples can be drawn efficiently.

We recall that given an arbitrary positive measure $\nu$ on $\Mc$, the space $L_2(\nu)$ is the Hilbert space of all measurable functions $\phi:\Mc\to \Cb$ satisfying 
\[
\| \phi\|_{L_2(\nu)} := \sqrt{\int_{\Mc} |\phi(x)|^2\,d\nu(x)} < \infty.
\]
Assuming $\Fc_N $ is a closed subspace of $L_2(\nu)$, the $L_2(\nu)$-projection of a function $\phi\in L_2(\nu)$ onto  $\Fc_N \subset L_2(\nu)$ is defined by
\begin{equation}\label{eq:projDef}
P_N^\nu \phi = \argmin_{f\in \Fc_N} \|f - \phi\|_{L_2(\nu)}= \argmin_{f\in \Fc_N} \int_{\Mc} |f - \phi|^2\,d\nu =\argmin_{c\in \Cb^N} \int_{\Mc} |c\herm \bs\psi - \phi|^2\,d\nu.
\end{equation}

Now, given the data points $x_1,\ldots,x_M$ from~(\ref{eq:data}), we define the empirical measure $\hat{\mu}_M$ by
\begin{equation}\label{eq:emp_meas_def}
\hat\mu_M = \frac{1}{M}\sum_{i=1}^M \delta_{x_i},
\end{equation}
where $\delta_{x_i}$ is the Dirac measure at $x_i$. In particular, the integral of a function $\phi$ with respect to $\hat{\mu}_M$ is given by
\[
\int_\Mc \phi(x)d\hat\mu_M(x) = \frac{1}{M}\sum_{i=1}^M \phi(x_i).
\]

We remark that the EDMD subspace $\Fc_N$ defined in~(\ref{eq:FNdef}) is a closed subspace of both $L_2(\hat{\mu}_M)$ and $L_2(\mu)$ and hence the projections $P_N^{\mu}$ and $P_N^{\hat\mu_M}$ are well defined.
}

Now we are ready to state the  following characterization of $\Kc_{N,M} \phi$:

\begin{theorem}\label{thm:L2proj}
Let $\hat \mu_M$ denote the empirical measure ~(\ref{eq:emp_meas_def}) associated to the sample points $x_1,\ldots, x_M$ and assume that the $N\times N$ matrix
\begin{equation}\label{eq:Mmuhat_def}
M_{\hat\mu_M}  =\frac{1}{M}\sum_{i=1}^M \bs\psi(x_i)\bs\psi(x_i)\herm = \int_{\Mc} \bs\psi \bs\psi\herm\,d\hat\mu_M 
\end{equation}
is invertible. Then for any $\phi \in \Fc_N$
\begin{equation}\label{eq:KoopL2proj}
\Kc_{N,M}\phi = P_{N}^{\hat{\mu}_M} \Kc \phi = \argmin_{f \in \Fc_N} \| f - \Kc \phi  \|_{L_2(\hat\mu_M )} ,
\end{equation}
i.e.,
\begin{equation}\label{eq:KoopL2_altern}
\Kc_{N,M} = P_N^{\hat\mu_M} \Kc_{|\Fc_N},
\end{equation}
where $\Kc_{|\Fc_N}:\Fc_N\to\Fc$ is the restriction of the Koopman operator to $\Fc_N$.
\end{theorem}
\begin{proof}
Since the matrix $M_{\hat\mu_M}$ is invertible, the least-squares problem~(\ref{eq:ls}) has a unique solution given~by
\[
a_i = \Big(\sum_{j=1}^M \bs\psi(x_j)\bs\psi(x_j)\herm\Big)^{-1}\sum_{j=1}^M \bs\psi(x_j)\overline{ \psi_i(y_j)},
\]
where $a_i\herm \in \Cb^{1\times N}$ is the $i^{\mr{th}}$ row of $A_{N,M}$. Therefore
\[
A_{N,M}\herm = \Big(\sum_{j=1}^M \bs\psi(x_j)\bs\psi(x_j)\herm\Big)^{-1}\sum_{j=1}^M \bs\psi(x_j)\bs\psi(y_j)\herm.
\]
On the other hand, analyzing the minimization problem on the right-hand side of~(\ref{eq:KoopL2proj}), we get for any $\phi = c_\phi\herm \bs \psi$
\[
 \argmin_{f \in \Fc_N} \| f - \Kc \phi  \|_{L_2(\hat\mu_M )}  = \argmin_{c\in \Cb^N}  \frac{1}{M}\sum_{i=1}^M ( c\herm \bs\psi(x_i) - c_\phi\herm\bs\psi(y_i))^2  
\]
with the unique minimizer (since the minimized functions is strictly convex in $c$)
\[
c = \Big(\sum_{j=1}^M \bs\psi(x_j)\bs\psi(x_j)\herm\Big)^{-1}\sum_{j=1}^M \bs\psi(x_j)\bs\psi(y_j)\herm c_\phi = A_{N,M}\herm c_\phi.
\]
Hence,  $\argmin_{f \in \Fc_N} \| f - \Kc \phi  \|_{L_2(\hat\mu_M )} = c\herm\bs\psi= c_\phi\herm A_{N,M}\bs\psi = \Kc_{N,M}\phi $ as desired, where we used~(\ref{eq:KNM_def}) in the last equality.
\end{proof}

Theorem~\ref{thm:L2proj} says that for any function $\phi\in\Fc_N$, the EDMD operator $\Kc_{N,M}$ computes the $L_2(\hat\mu_M)$-orthogonal projection of $\Kc\phi$ on the subspace spanned by $\psi_1,\ldots,\psi_N$.

\begin{remark}If the assumption that the matrix $M_{\hat\mu_M}$ is invertible does not hold, then the solution to the projection problem on the right-hand side of~(\ref{eq:KoopL2proj}) may not be unique\footnote{To be more specific, if the matrix $M_{\hat\mu_M}$ is not invertible, the solution to (\ref{eq:KoopL2proj}) may not be unique when viewed as a member of $ L_2(\mu)$. When viewed as a member of $L_2(\hat\mu_M)$, the solution is unique (since it is a projection onto a closed subspace of a Hilbert space). This is because in this case two functions from $\Fc_N$ belonging to distinct $L_2(\mu)$ equivalence classes may fall into the same $L_2(\hat\mu_M)$ equivalence class.}. The action of the EDMD operator $\Kc_{N,M}$ (which is defined uniquely by~(\ref{eq:ANM}) and (\ref{eq:KNM_def})) then selects one solution to the projection problem. In concrete terms, we have
\[
\Kc_{N,M}\phi\in \Argmin_{f\in\Fc_N}\| f - \Kc \phi\|_{L_2(\hat\mu_M)},
\]
where $\Argmin_{f\in\Fc_N}\| f - \Kc \phi\|_{L_2(\hat\mu_M)}$ denotes the set of all minimizers of $\|f - \Kc\phi \|_{L_2(\hat\mu_M)}$ among $f\in\Fc_N$.
\end{remark}

\section{Convergence of $\Kc_{N,M}$ as $M\to\infty$}
\label{EDMDconv}
The first step in understanding convergence of EDMD is to understand the convergence of $\Kc_{N,M}$ as the number of samples $M$ tends to infinity. In this section we prove that \[
\Kc_{N,M} \to \Kc_N,
\]
 where
\begin{equation}
\label{eq:KNdef}
\Kc_N = P_N^\mu \Kc_{|\Fc_N},
\end{equation}
provided that the samples $x_1,\ldots,x_M$ are drawn independently from a given probability distribution $\mu$ (e.g., uniform distribution for compact $\Mc$ or Gaussian for if $\Mc = \Rb^n$).

\begin{assumption}[$\mu$ independence]\label{as:indep}
The basis functions $\psi_1,\ldots,\psi_N$ are such that \[\mu\{x\in \Mc \mid c\herm \bs \psi(x)=0 \} = 0\] for all nonzero $c
\in\Cb^N$.
\end{assumption}
This is a natural assumption ensuring that the measure $\mu$ is not supported on a zero level set of a linear combination of the basis functions used. It is satisfied if $\mu$ is any measure with the support equal to $\Mc$ in conjunction with the majority of most commonly used basis functions such as polynomials, radial basis functions with unbounded support (e.g., Gaussian, thin plate splines) etc. This assumption in particular implies that the matrix $M_{\hat\mu_M}$ defined in~(\ref{eq:Mmuhat_def}) is invertible with probability one for $M \ge N$ if $x_j$'s are iid samples from~$\mu$.

\begin{lemma}\label{lem:projConv}
If Assumption~\ref{as:indep} holds, then for any $\phi \in \Fc$ we have with probability one
\begin{equation}
\lim_{M\to\infty}\|P_N^{\hat\mu_M} \phi -  P_N^{\mu} \phi\| = 0,
\end{equation}
where $\|\cdot \|$ is any norm on $\Fc_N$ (which are all equivalent since $\Fc_N$ is finite dimensional).
\end{lemma}
\begin{proof}
We have
\begin{align*}
P_N^{\mu} \phi &= \argmin_{f\in\Fc_N} \int_{\Mc} |f - \phi|^2\, d\mu= \bs\psi\herm\argmin_{c\in\Cb^N} \int_{\Mc} |c\herm \bs \psi - \phi|^2\, d\mu \\ &=\bs\psi\herm\argmin_{c\in\Cb^N} \big[c\herm M_{\mu} c - 2\Re\{c\herm b_{\mu,\phi}\}   \big],
\end{align*}
where
\[
M_\mu  = \int_{\Mc} \bs\psi \bs\psi\herm\,d\mu \in \Cb^{N\times N},\quad b_{\mu,\phi} = \int_{\Mc}\bs\psi\overline{\phi}\,d\mu \in \Cb^N
\]
and we dropped the constant term in the last equality which does not influence the minimizer. By Assumption~\ref{as:indep}, the matrix $M_{\mu}$ is invertible and hence Hermitian positive definite. Therefore the unique minimizer is $c = M_{\mu}^{-1}b_{\mu,\phi}$. Hence
\[
P_N^{\mu}\phi = b_{\mu,\phi}\herm M_{\mu}^{-1}\bs\psi.
\]
On the other hand, the same computation shows that
\[
P_N^{\hat\mu_M}\phi = b_{\hat\mu_M,\phi}\herm M_{\hat\mu_M}^{-1}\bs\psi
\]
with 
\[
 b_{\hat\mu_M,\phi} = \int_{\Mc}\bs\psi \overline{\phi}\,d\hat\mu_M = \frac{1}{M}\sum_{i=1}^M \bs\psi(x_i)\overline{\phi(x_i)}
\]
and with the matrix $M_{\hat\mu_M}$ defined in~(\ref{eq:Mmuhat_def}) guaranteed to be Hermitian positive definite by Assumption~\ref{as:indep} with probability one for $M\ge N$. The result then follows by the strong law of large numbers which ensures that \[\lim_{M\to\infty} (b_{\hat\mu_M,\phi}\herm M_{\hat\mu_M}^{-1}) = b_{\mu,\phi}\herm M_{\mu}^{-1}\]
with probability one since the matrix function $A\mapsto A^{-1}$ is continuous and the samples $x_i$ are iid.
\end{proof}

\begin{theorem}\label{thm:convFin}
If Assumption~\ref{as:indep} holds, then we have with probability one for all $\phi \in \Fc_N$
\begin{equation}\label{eq:opConvStrongFin}
\lim_{M\to\infty}\|\Kc_{N,M} \phi -  \Kc_N \phi\| = 0,
\end{equation}
where $\| \cdot\|$ is any norm on $\Fc_N$. In particular
\begin{equation}\label{eq:opConvNormFin}
\lim_{M\to\infty}\|\Kc_{N,M} -  \Kc_N \| = 0,
\end{equation}
where $\| \cdot\|$ is any operator norm and
\begin{equation}\label{eq:opConvSpecFin}
\lim_{M\to\infty}\mr{dist}\big(\sigma(\Kc_{N,M}),\sigma(\Kc_N)\big) = 0,
\end{equation}
where $\sigma(\cdot) \subset \mathbb{C}$ denotes the spectrum of an operator and $\mr{dist}(\cdot,\cdot)$ the Hausdorff metric on subsets of $\mathbb{C}$.
\end{theorem}
\begin{proof}
For any fixed $\phi \in\Fc_N$ we have by Theorem~\ref{thm:L2proj} 
\[
\Kc_{N,M} \phi = P_{N}^{\hat\mu_M} \Kc \phi = P_{N}^{\hat\mu_M} (\phi\circ T). \]
By definition of $\Kc_N \phi$ we have $\Kc_N \phi =P_N^\mu (\phi\circ T) $ and therefore~(\ref{eq:opConvStrongFin}) holds from Lemma~\ref{lem:projConv} with probability one. Since $\Fc_N$ is finite dimensional, (\ref{eq:opConvStrongFin}) holds with probability one for all basis functions of $\Fc_N$ and hence by linearity for all $\phi\in\Fc_N$. Convergence in the operator norm~(\ref{eq:opConvNormFin}) and spectral convergence~(\ref{eq:opConvSpecFin}) follows from~(\ref{eq:opConvStrongFin}) since the operators $\Kc_{N,M}$ and $\Kc_N$ are finite dimensional.
\end{proof}

Theorem~\ref{thm:convFin} tells us that in order to understand the convergence of $\Kc_{N,M}$ to $\Kc$ it is sufficient to understand the convergence of $\Kc_N$ to $\Kc$. This convergence is analyzed in Section~\ref{KNconv}.

\subsection{Ergodic sampling}\label{sec:ErgodSamp}
The assumption that the samples $x_1,\ldots,x_M$ are drawn independently from the distribution~$\mu$ can be replaced by the assumption that $(T,\Mc,\mu)$ is ergodic and the samples $x_1,\ldots,x_M$ are the iterates of the the dynamical system starting from some initial condition $x\in\Mc$, i.e., $x_i = T^i(x)$. Provided that Assumption~\ref{as:indep} holds, both Lemma~\ref{lem:projConv} and Theorem~\ref{thm:convFin} hold without change; the statement ``with probability one'' is now interpreted with respect to drawing the initial condition $x$ from the distribution $\mu$. The proofs follow exactly the same argument, only the strong law of large numbers is replaced by the Birkhoff's ergodic theorem in Lemma~\ref{lem:projConv}.

\section{Convergence of $\Kc_N$ to $\Kc$}
\label{KNconv}
In this section we investigate convergence of $\Kc_N$ to $\Kc$ in the limit as $N$ goes to infinity. Since the operator $\Kc_N$ is defined on $\Fc_N$ rather than $\Fc$, we extend the operator to all of $\Fc$ by precomposing with $P_N^\mu$, i.e., we study the convergence of $\Kc_N P_N^\mu = P_N^\mu\Kc P_N^\mu:\Fc\to\Fc $ to $\Kc:\Fc\to\Fc$ as $N\to\infty$. Note that as far as spectrum is concerned, precomposing with $P_N^\mu$ just adds a zero to the spectrum.

\new{
To simplify notation, in what follows we denote the $L_2(\mu)$ norm of a function $f$ by
 \[
\| f\| := \| f\|_{L_2(\mu)} = \sqrt{\int_{\Mc} |f|^2\,d\mu}  
\]
and
the usual $L_2(\mu)$ inner product by
\[
\langle f,g\rangle := \int_{\Mc} f\overline g\,d\mu.  
\]

\subsection{Preliminaries}

Before stating our results, we recall several concepts from functional analysis and operator theory.

\begin{definition}[Bounded operator]
An operator $\Ac:\Fc\to\Fc$ defined on a Hilbert space $\Fc$ is bounded if
\[
\| \Ac \| := \sup\limits_{f \in \Fc,\; \| f \|=1} \| \Ac f \| < \infty.
\]
The quantity $\| \Ac \| $ is referred as the norm of $\Ac$.
\end{definition}

\begin{definition}[Convergence in strong operator topology]\label{def:strongConv}
A sequence of bounded operators $\Ac_i: \Fc \to \Fc$  defined on a Hilbert space $\Fc$ convergences strongly (or in the strong operator topology) to an operator $\Ac : \Fc\to \Fc$ if 
\begin{equation}\label{eq:strongConvDef}
 \lim_{i\to\infty}  \| \Ac_i g - \Ac g\| 
\end{equation}
for all $g \in \Fc$.
\end{definition}

\begin{definition}[Weak convergence]\label{def:weakConv}
A sequence of elements $f_i \in \Fc$ of a Hilbert space $\Fc$ converges weakly to $f \in \Fc$, denoted $f_i \tow f$, if 
\begin{equation}\label{eq:weakConvDef}
 \lim_{i\to\infty}  \langle f_i, g\rangle  = \langle f,g\rangle
\end{equation}
for all $g \in \Fc$.
\end{definition}
We emphasize that Definition~\ref{def:strongConv} pertains to convergence of operators defined on $\Fc$ whereas Definition~\ref{def:weakConv} pertains to convergence of elements of $\Fc$. We also remark that strong convergence of $f_i \to f$ (i.e., $\|f_i - f\|\to 0$) implies weak convergence but not vice versa. Similarly, convergence in the strong operator topology implies convergence in the weak operator topology (i.e, $\Ac_i g \tow \Ac_g$ for all $g$) but does not imply convergence in the operator norm (i.e., $\|\Ac_i - \Ac \| \to 0$).

In our setting of $\Fc = L_2(\mu)$, the statements~(\ref{eq:strongConvDef}) and~(\ref{eq:weakConvDef}) translate to the requirements that, respectively, 
\[
\lim_{i\to\infty}  \sqrt{\int_{\Mc} | \Ac_i g - \Ac g |^2\,d\mu} = 0\qquad \mr{and}\qquad  \lim_{i\to\infty}  \int_{\Mc} f_i\overline g \,d\mu = \int_{\Mc} f\overline g \,d\mu
\]
for all $g\in L_2(\mu)$.
}

For the remainder of this work, we invoke the following assumption:
\begin{assumption}\label{as:bounded}
The following conditions hold:
\begin{enumerate}
\item The Koopman operator $\Kc:\Fc \to \Fc$ is bounded.
\item The observables $\psi_1,\ldots,\psi_N$ defining $\Fc_N$ are selected from a given orthonormal basis of $\Fc$, i.e., $(\psi_i)_{i=1}^\infty$ is an orthonormal basis of $\Fc$.
\end{enumerate}
\end{assumption}
The first part of the assumption holds for instance when $T$ is invertible, Lipschitz with Lipschitz inverse and $\mu$ is the Lebesgue measure on $\Mc$ (or any measure absolutely continuous w.r.t. the Lebesgue measure with bounded and strictly positive density). The second part of the assumption is non-restrictive as any countable dense subset of $\Fc$ can be turned into an orthonormal basis using the Gram-Schmidt  process.

\subsection{Convergence in strong operator topology}
In this section we prove convergence in the strong operator topology (Definition~\ref{def:strongConv}) of $\Kc_N P_N^\mu$ to $\Kc$. First, we need the following immediate lemma:

\begin{lemma}\label{lem:projStrong}
If $(\psi_i)_{i=1}^\infty$ form an orthonormal basis of $\Fc = L_2(\mu)$, then $P_N^\mu$ converge strongly to the identity operator $I$ and in addition $\| I - P_N\| \le 1 $ for all $N$.
\end{lemma}
\begin{proof}
Let $\phi = \sum_{i=1}^\infty c_i\psi_i$ with $\| \phi\| = 1$. Then by Parseval's identity $\sum_{i=1}^\infty |c_i|^2 = 1$ and
 \[\|P_N^\mu \phi - \phi \| = \Bigg\| \sum_{i=N+1}^\infty c_i\psi_i \Bigg\| =  \sum_{i=N+1}^\infty |c_i|^2 \to 0\]
with $\sum_{i=N+1}^\infty |c_i|^2 \le 1$ for all $N$.
\end{proof}

Now we are ready to prove strong convergence of $P_N^\mu\Kc P_N^\mu$ to $\Kc$. 
\begin{theorem}\label{thm:strongConv}
If Assumption~\ref{as:bounded} holds, then the sequence of operators $\Kc_NP_N^\mu=P_N^\mu \Kc P_N^\mu $ converges strongly to $\Kc$ as $N\to \infty$, i.e.,
\[
\lim_{N\to\infty} \int_{\Mc} | \Kc_N P_N^\mu \phi - \Kc\phi |^2\,d\mu  = 0
\]
for all $\phi \in \Fc$.
\end{theorem}
\begin{proof}
Let $\phi \in \Fc$ be given. Then, writing $\phi=P_N^\mu \phi + (I -P_N^\mu)\phi$ we have
\begin{align*}
\| P_N^\mu \Kc P_N^\mu \phi - \Kc \phi \| &= \| (P_N^\mu -I)\Kc P_N^\mu \phi + \Kc(P_N^\mu-I)\phi    \| \le \| (P_N^\mu -I)\Kc P_N^\mu \phi  \|  + \|\Kc \|  \|( I - P_N^\mu)\phi    \|  \\
& \le \| (P_N^\mu -I)\Kc \phi  \| + \| (P_N^\mu -I)\| \| \Kc P_N^\mu \phi -\Kc  \phi  \| +  \|\Kc \|  \|( I - P_N^\mu)\phi    \| \to 0
\end{align*}
by Lemma~\ref{lem:projStrong} and by the fact that $\Kc P_N^\mu\phi \to \Kc\phi$ since $\Kc$ is continuous by Assumption~\ref{as:bounded}.
\end{proof}

\subsection{Weak spectral convergence}

Unfortunately, strong converge does not in general guarantee convergence of the spectra of the operators. This is guaranteed only if the operators converge in the operator norm\footnote{A sequence of operators $\Ac_i$ converges in the operator norm to an operator $\Ac$ if $\lim_{i\to\infty}\|\Ac_i -\Ac\| = 0 $.}. \new{An important exception to this is the case of $\Fc_N$ being an invariant subspace, i.e., $\Kc f \in \Fc_N$ for all $f\in\Fc_N$ in which case the spectra of $\Kc_N$ and $\Kc_{|\Fc_N}$ coincide. Here, however, we \emph{do not assume} that $\Fc_N$ is invariant and  prove certain spectral convergence results in a weak sense}. In particular, we prove convergence of the eigenvalues of $\Kc_N$ along a subsequence and weak convergence of the associated eigenfunctions (see Definition~\ref{def:weakConv}), provided that the weak limit of the eigenfunctions is nonzero. 

\begin{theorem}\label{thm:weakConv}
If Assumption~\ref{as:bounded} holds and $\lambda_N$ is a sequence of eigenvalues of $\Kc_N$ with the associated normalized eigenfunctions $\phi_N \in \Fc_N$, $\|\phi_N  \|=1$, then there exists a subsequence $(\lambda_{N_i},\phi_{N_i})$ such that 
\[
\lim_{i\to\infty}\lambda_{N_i} = \lambda,\quad \phi_{N_i} \tow \phi,
\]
where $\lambda \in \mathbb{C}$  and $\phi\in \Fc$ are such that $\Kc\phi = \lambda \phi$. In particular if $\| \phi\| \ne 0$, then $\lambda$ is an eigenvalue of $\Kc$ with eigenfunction $\phi$.
\end{theorem}
\begin{proof}
First, observe that since $\Kc_N \phi_N = \lambda_N \phi_N$ with $\phi_N \in \Fc_N$, we also also have  $P_N^\mu \Kc P_N^\mu \phi_N =  \lambda_N \phi_N $. Hence $|\lambda_N|  \le  \| P_N^\mu \Kc P_N^\mu   \| \le \| \Kc\| < \infty$ by Assumption~\ref{as:bounded} and the fact that $\| P_N^\mu\| \le 1 $. Therefore the sequence $\lambda_N$ is bounded. Since $\phi_N$ is normalized and hence bounded, by weak sequential compactness of the unit ball of a Hilbert space (which follows from the Banach-Alaoglu theorem~\cite[Theorems 3.15]{rudinFA} and Eberlein-\v Smulian theorem~\cite{dunford_schwartz}), there exists a subsequence $(\lambda_{N_i},\phi_{N_i})$ such that $\lambda_{N_i}\to\lambda $ and $\phi_{N_i}\tow \phi $.

It remains to prove that $(\lambda,\phi)$ is an eigenvalue-eigenfunction pair of $\Kc$. For ease of notation, set $\lambda_{i} = \lambda_{N_i}$ and $\phi_i  = \phi_{N_i}$. Denote $\hat\Kc_i = \Kc_{N_i} P_{N_i}^\mu = P_{N_i}^\mu \Kc P_{N_i}^\mu$ and observe that $\hat\Kc_i \phi_i = \lambda_i\phi_i$ for all $i$. Then we have
\begin{align*}
  \Kc\phi =  \hat\Kc_i(\phi-\phi_i) + (\Kc-\hat\Kc_i)\phi + \hat\Kc_i\phi_i.
\end{align*}
Taking the inner product with an arbitrary $f\in\Fc$ and using the fact that $\hat\Kc_i\phi_i = \lambda_i\phi_i$, we get
\[
 \langle \Kc\phi,f \rangle =  \langle\hat\Kc_i(\phi-\phi_i),f \rangle+\langle(\Kc-\hat\Kc_i)\phi,f\rangle + \langle\lambda_i\phi_i ,f \rangle.
\]
Now, the second term on the right hand side $\langle(\Kc-\hat\Kc_i)\phi,f\rangle\to 0$ since $\hat\Kc_i$ converges strongly to $\Kc$ by Theorem~\ref{thm:strongConv}. The last term $\langle\lambda_i\phi_i ,f \rangle \to \langle \lambda\phi,f\rangle$ since $\lambda_i\to\lambda$ and $\phi_i \tow \phi$. It remains to show that the first term converges to zero. We have
\[\langle \hat\Kc_i(\phi-\phi_i),f \rangle = \langle P_{N_i}^\mu\Kc P_{N_i}^\mu(\phi-\phi_i),f \rangle = \langle \Kc (P_{N_i}^\mu\phi-\phi_i), P_{N_i}^\mu f \rangle,    \] 
where we used the fact that $P_{N_i}^\mu$ is self-adjoint and $\phi_i \in \Fc_{N_i}$ and hence $P_{N_i}^\mu \phi_i = \phi_i$. Denote $h_i := \Kc (P_{N_i}^\mu\phi-\phi_i)$. We will show that $h_i\tow 0$. Indeed, denoting $\Kc^\star$ the adjoint of $\Kc$, we have
\[
\langle \Kc (P_{N_i}^\mu\phi-\phi_i),f\rangle = \langle (P_{N_i}^\mu\phi-\phi + \phi -\phi_i),\Kc^\star f\rangle = \langle P_{N_i}^\mu\phi-\phi ,\,\Kc^\star f\rangle + \langle \phi -\phi_i,\,\Kc^\star f\rangle \to 0,
\]
since $P_{N_i}^\mu$ converges strongly to the identity (Lemma~\ref{lem:projStrong}) and $\phi_i \tow \phi$. Finally, we show that $\langle h_i, P_{N_i}^\mu f\rangle \to 0$. We have
\[
\langle h_i, P_{N_i}^\mu f\rangle  = \langle h_i, P_{N_i}^\mu f - f\rangle + \langle h_i, f\rangle. 
\]
The second term goes to zero since $h_i\tow 0$. For the first term we have
\[
\langle h_i, P_{N_i}^\mu f - f\rangle \le \| h_i\| \|P_{N_i}^\mu f - f  \| \to 0
\]
since $P_{N_i}^\mu$ converges strongly to the identity operator (Lemma~\ref{lem:projStrong}) and $h_i$ is bounded since $\Kc$ is bounded by Assumption~\ref{as:bounded}, $\| P_{N_i}^\mu\| \le 1$ and $\| \phi_i \| \le 1$.
Therefore we conclude that
\[
 \langle \Kc\phi,f \rangle = \lim_{i\to\infty} \langle \lambda_i \phi_i,f\rangle =  \langle \lambda \phi ,f \rangle 
\]
for all $f \in \Fc$. Therefore $\Kc \phi = \lambda \phi$.
\end{proof}

\textbf{Example}\;\; {\it As an example demonstrating that the assumption that the weak limit of $\phi_N$ is nonzero is important, consider $\Mc = [0,1]$, $T(x) = x$ and $\mu$ the Lebesgue measure on $[0,1]$. In this setting, the Koopman operator $\Kc: L_2(\mu)\to L_2(\mu)$ is the identity operator with the spectrum being the singleton $\sigma(\Kc)=\{1\}$. However, given \emph{any} $\lambda \in \mathbb{C}$ and the sequence of functions $\phi_N = \sqrt{2}\sin(2\pi N x)$, we have \[\Kc\phi_N - \lambda\phi_N = \phi_N - \lambda\phi_N  = (1-\lambda)\sqrt{2}\sin(2\pi N x)\tow 0\]
with $\| \phi_N\|^2 = \int_{0}^1 2\sin^2(2\pi N x)\,dx = 1$. Therefore if $\phi_N$ were the eigenfunctions of $\Kc_N$ with eigenvalues $\lambda_N\to \lambda \ne 1$, then the sequence $\lambda_N$ would converge to a spurious eigenvalue $\lambda$. Fortunately, in this case, we have $\sigma(\Kc_N) = \{ 1\}$ and hence no spurious eigenvalues exist; however, in general, we cannot rue out this possibility, at least not as far as the statement of  Theorem~\ref{thm:weakConv} goes. See Figure~\ref{fig:sinus} for illustration.

This example, with the highly oscillatory functions $\phi_N$, may motivate practical considerations in detecting spurious eigenvalues, e.g., using Sobolev type (pseudo) norms $\int_{\Mc} \| \nabla \phi_N\|^2\,d\mu$ or other metrics of oscillatoriness. See, e.g., \cite{giannakis2016data} for the use of Sobolev norms in the context of Koopman data analysis and forecasting.}

\begin{figure*}[th]
\begin{picture}(140,170)
\put(130,0){\includegraphics[width=80mm]{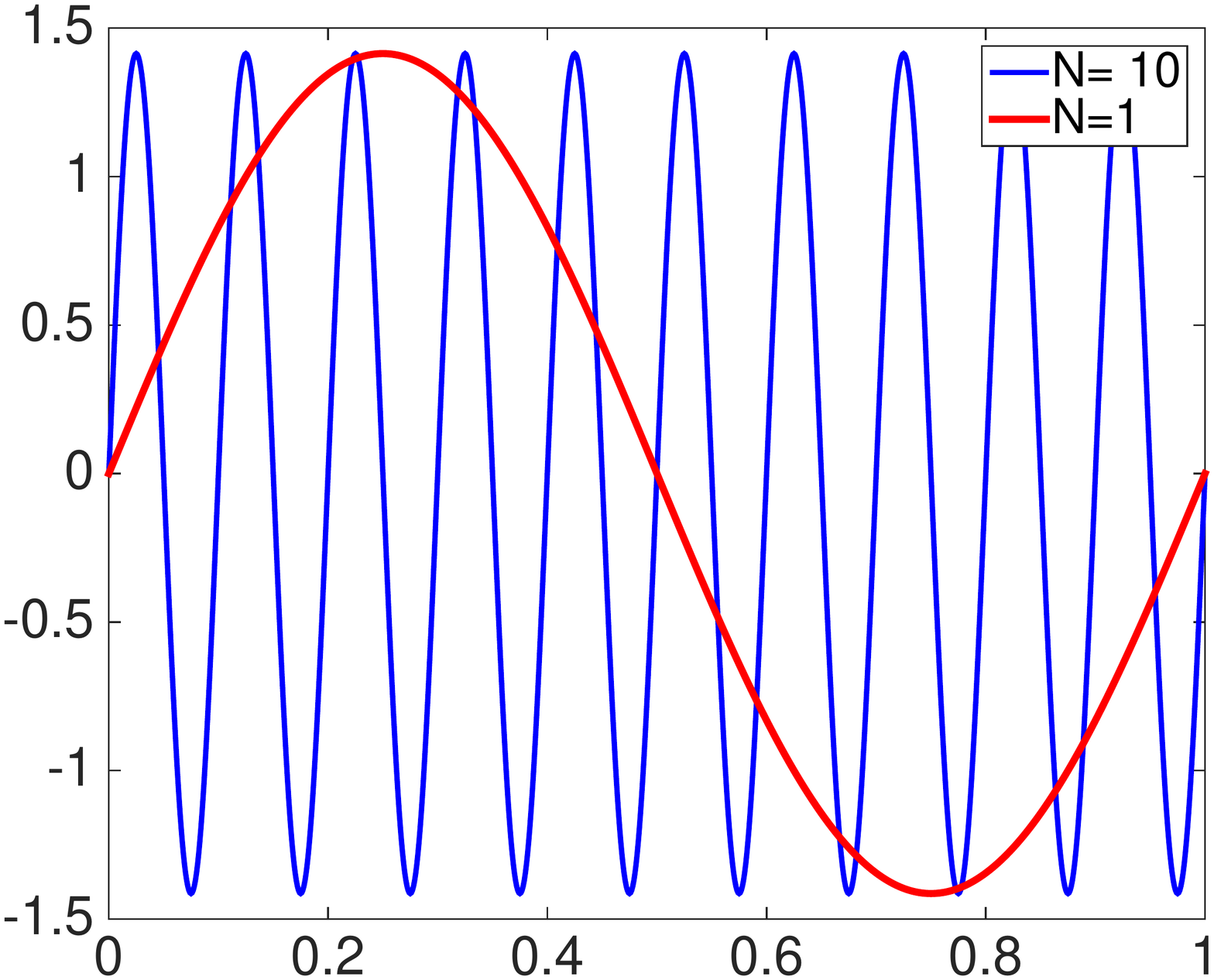}}
\put(245,0){$x$}
\put(105,90){$\phi_N(x)$}
\end{picture}
\caption{Graph of the functions $\phi_N(x) = \sqrt{2}\sin(2\pi Nx)$ for $N = 1$ and $N=10$. These functions satisfy $\| \phi_N\|_{L_2}=1$ and $\phi_N\tow 0$ as $N\to\infty$. If such $\phi_N$ happen to be eigenfunctions of $\Kc_N$ with eigenvalues $\lambda_N$, then the accumulation points of the sequence $(\lambda_N)_{N=1}^\infty$ need not be eigenvalues of the Koopman operator.}
\label{fig:sinus}
\end{figure*}


As an immediate corollary of Theorem~\ref{thm:weakConv}, we get:
\begin{corollary}\label{thm:weakConv_NM}
If Assumption~\ref{as:bounded} holds and $\lambda_{N,M}$ is a sequence of eigenvalues of $\Kc_{N,M}$ with the associated normalized eigenfunctions $\phi_{N,M} \in \Fc_N$, $\|\phi_{N,M}  \|=1$, then there exists a subsequence $(\lambda_{N_i,M_j},\phi_{N_i,M_j})$ such that with probability one 
\[
\lim_{i\to\infty}\lim_{j\to\infty}\lambda_{N_i,M_j} = \lambda,\quad  \lim_{i\to\infty}\lim_{j\to\infty}\langle \phi_{N_i,M_j},f\rangle = \langle \phi,f\rangle,
\]
for all $f \in \Fc$, where $\lambda \in \mathbb{C}$  and $\phi\in \Fc$ are such that $\Kc\phi = \lambda \phi$. In particular if $\| \phi\| \ne 0$, then $\lambda$ is an eigenvalue of $\Kc$ with eigenfunction $\phi$.

\end{corollary}
\begin{proof}
First notice that since $\Kc_{N,M} \to \Kc_N$ in the operator norm (Theorem~\ref{thm:convFin}) and $\|\Kc_N\| \le \|\Kc \| < \infty$, the sequence $\lambda_{N,M}$ is bounded. Since $\phi_{N,M}$ are normalized and hence bounded,  we can extract a subsequence $(\lambda_{N,M_j},\phi_{N,M_j})$ such that $\lim_{j\to\infty} \lambda_{N,M_j}=\lambda_N \in \mathbb{C}$ and $\lim_{j\to\infty} \phi_{N,M_j} = \phi_N \in \Fc_N$ (strong convergence as $\Fc_N$ is finite dimensional). Then
\[
\Kc_N\phi_N = (\Kc_N - \Kc_{N,M_j})\phi_N + \Kc_{N,M_j}(\phi_N-\phi_{N,M_j}) + \Kc_{N,M_j}\phi_{N,M_j}.
\]
Since $\Kc_{N,M_j}$ converges strongly to $\Kc_N$ with probability one (Theorem~\ref{thm:convFin}) and since $\phi_{N,M_j}$ converges strongly to $\phi_N$, the first two terms go to zero with probability one as $j$ tends to infinity. The last term is equal to $\lambda_{N,M_j} \phi_{N,M_j}$ and hence necessarily $\Kc\phi_N = \lambda_N \phi_N$, $\| \phi_N\| =1$, with probability one. The result then follows from Theorem~\ref{thm:weakConv}.
\end{proof}

\section{Implications for finite-horizon predictions}\label{sec:finiteHor}
One of the main roles of an approximation to the Koopman operator is to provide a \emph{prediction} of the evolution of a given observable. Whereas obtaining accurate predictions over an infinite-time horizon cannot be expected in general from the EDMD approximation of the Koopman operator, a prediction over any \emph{finite horizon} is asymptotically, as $N\to\infty$, exact when the prediction error is measured in the $L_2(\mu)$ norm:
\begin{theorem}\label{thm:predFin}
Let $f\in \Fc^n$ be a given (vector) observable\footnote{We choose to state the theorem for vector observables as this is the form of prediction typically encountered in practice. For a vector observable $f\in \Fc^n$, the norm $\| f\|$ is defined by $\sum_{i=1}^n \|f_i \|_{L_2(\mu)}, $ where $f_i\in\Fc$ is the $i^{\mr{th}}$ component of $f$.} and let Assumption~\ref{as:bounded} hold. Then for any $\Omega \in \mathbb{N}$ we have
\begin{equation}\label{eq:predFin_gen}
\lim_{N\to\infty} \sup_{i\in\{1,\ldots,\Omega\}} \| (\Kc_N)^i P_N^\mu f - \Kc^i f  \| = 0.
\end{equation}
In particular, if $f\in \Fc_{N_0}^n$ for some $N_0 \in \mathbb{N}$, then \begin{equation}\label{eq:predFin_spec}
\lim_{N\to\infty} \sup_{i\in\{1,\ldots,\Omega\}} \| (\Kc_N)^i  f - \Kc^i f  \| = 0.
\end{equation}
\end{theorem}
\begin{proof}
We proceed by induction. Let $f\in \Fc$. For $\Omega=1$, the result is exactly Theorem~\ref{thm:strongConv}. Let the result hold form some $\Omega \in \mathbb{\Nb}$. It is sufficient to prove that $\| (\Kc_N)^{\Omega+1} P_N^\mu f - \Kc^{\Omega+1} f  \| \to 0$ as $N\to \infty$. We have
\begin{align*}
\| (\Kc_N)^{\Omega+1} P_N^\mu f - \Kc^{\Omega+1} f  \| & = \| \Kc_N(\Kc_N)^{\Omega} P_N^\mu f - \Kc\Kc^{\Omega} f  \|  = \| \Kc_N g_N - \Kc g  \| \\
&\le \| \Kc_N g- \Kc g \| + \| \Kc_N(g_N - g) \| \le  \| \Kc_N g- \Kc g \| + \| \Kc\| \| g_N - g \|.
\end{align*}
where $g_N = (\Kc_N)^{\Omega} P_N^\mu f$ and $g = \Kc^{\Omega} f$; in the last inequality we used the fact that $\|\Kc_N\| \le \| K\|$. The term  $\| \Kc_N g- \Kc g \| $ tends to zero by Theorem~\ref{thm:strongConv}, whereas the term $\| g_N - g \| \to 0$ by the induction hypothesis. This proves~(\ref{eq:predFin_gen}) for a scalar observable $f\in \Fc$. The general result with a vector valued observable $f\in \Fc^n$ follows by applying the the same reasoning to each component of $f$. The result~(\ref{eq:predFin_spec}) follows from (\ref{eq:predFin_gen}) since if $f\in \Fc_{N_0}^n$ for some $N_0\in\mathbb{N}$, then $P_N^\mu f = f$ for $N\ge N_0$.
\end{proof}

To be more specific on practical use of $\Kc_N$ for prediction, assume that $f\in \Fc_{N_0}^n$ for some $N_0\in\mathbb{N}$. Then for all $N\ge N_0$ there exists a matrix $C_N\in \Rb^{n
\times N}$ such that $f = C_N\bs\psi_N$, where $\bs\psi_N = [\psi_1,\ldots,\psi_N]^\top$ are the observables used in EDMD. Assume that an initial state $x_0$ is given and the values of the observables $\bs \psi_N(x_0)$ are known and we wish to predict the value of the observable $f$ at a state $x_i = T^i(x_0)$, i.e., $i$ steps ahead in the future. Using $\Kc_N$, this prediction is given by $C_N A_N^i \bs\psi_N(x_0)$, where \[A_N = \lim_{M\to\infty} A_{N,M}\] with $A_{N,M}$ defined\footnote{In Section~\ref{analyEDMD}, we show how the matrix $A_N$ can be constructed analytically.} in~(\ref{eq:ANM}). Theorem~\ref{thm:predFin} then says that 
\begin{equation}\label{eq:finitePred_specpec}
\lim_{N\to\infty} \int_{\Mc} \| C_N A_N^i \bs\psi_N   - f\circ T^i \|^2_2\,d\mu = 0\qquad \forall\, i\in\mathbb{N}.
\end{equation}
A typical application of Theorem~\ref{thm:predFin} is the prediction of the future state $x$ of the dynamical system~(\ref{eq:sys}) with a finite-dimensional state-space $\Mc\subset \Rb^n$. In this case, one simply sets $f(x) = x$. A crucial feature of the predictor obtained in this way is its \emph{linearity} in the ``lifted state'' $z=\bs\psi_N(x)$, allowing linear tools to address a nonlinear problem. This concept was succesfully applied to model predictive control in~\cite{korda2016linear} and to state estimation in~\cite{surana_estim}.

\begin{remark}
If $A_{N,M}$ is used instead of $A_N$ in Theorem~\ref{thm:predFin} and Equation~(\ref{eq:finitePred_specpec}), then the same convergence results hold with a double limit, first taking the number of samples $M$ to infinity and then the number of basis functions $N$. In particular, we get
\begin{equation}\label{eq:finitePred_specpec_NM}
\lim_{N\to\infty} \lim_{M\to\infty} \int_{\Mc} \| C_N A_{N,M}^i \bs\psi_N   - f\circ T^i \|^2_2\,d\mu = 0\qquad \forall\, i\in\mathbb{N}.
\end{equation}
\end{remark}

\section{Analytic EDMD}
\label{analyEDMD}
The results of the previous sections suggests a variation of the EDMD algorithm provided that the mapping $T$ is known in closed form and provided that the basis functions $\psi_i$ are such that the integrals of $\int_{\Mc} \psi_i \overline{\psi_j}\,d\mu$ and $\int_{\Mc} (\psi_i\circ T)\overline{\psi_j}\,d\mu$ can be computed analytically. This is the case in particular if $T$ and $\psi_i$'s are simple functions such as multivariate polynomials or trigonometric functions and $\mu$ is the uniform distribution over a simple domain $\Mc$ such as a box or a ball, or, e.g., a Gaussian distribution over $\Rb^n$.

Provided that such analytical evaluation is possible, one can circumvent the sampling step of EDMD and construct directly $\Kc_N$ rather than $\Kc_{N,M}$. Indeed, define
\begin{equation}
A_N = M_{T,\mu} M_{\mu}^{-1},
\end{equation}
where
\[
M_{\mu} = \int_{\Mc} \bs\psi \bs\psi\herm\,d\mu,\quad M_{T,\mu} = \int_{
\Mc}(\bs\psi\circ T)\bs\psi\herm\,d\mu.
\]
Then the operator from $\Fc_N$ to $\Fc_N$ defined by $c\herm \bs \psi\mapsto c\herm A_N\bs\psi$ is exactly $\Kc_N = P_N^\mu \Kc_{|\Fc_N}$.
\begin{theorem}
If the matrix $M_\mu$ is invertible, then for any $\phi = c_\phi\herm \bs\psi \in \Fc_N$ we have
\[
c_\phi\herm A_N\bs\psi = \Kc_N\phi.
\]
\end{theorem}
\begin{proof}
Given $\phi = c_\phi\herm \bs\psi$ we get
\[
\Kc_N\phi = P_N^\mu \Kc \phi = \bs\psi\herm\argmin_{c\in\Cb^N}\int_{\Mc}[c\herm\bs\psi - c_\phi\herm (\bs \psi\circ T)]^2\,d\mu = \bs\psi\herm\argmin_{c\in\Cb^N}\big[ c\herm M_{\mu} c - 2\Re\{c\herm M_{T,\mu}\herm c_{\phi}\}\big]
\]
with the unique minimizer $c = M_{\mu}^{-1}M_{T,\mu}\herm c_\phi$ (since $M_\mu$ is invertible, therefore Hermitian positive definite, and hence the minimized function is strictly convex). Therefore as desired \[\Kc_N\phi = c\herm\bs\psi = c_\phi\herm M_{T,\mu} M_{\mu}^{-1}\bs\psi = c_\phi\herm A_N \bs\psi. \] 
\end{proof}

\paragraph{Example} {\it In order to demonstrate the use of Analytic DMD, we compare the spectra of $\Kc_N$ and $\Kc_{N,M}$ for various values of $M$. The system considered is the logistic map
\[
x^+ = 2x^2 - 1,\quad x \in [-1,1].
\]
The measure $\mu$ is taken to be the uniform distribution on $[-1,1]$, which is not invariant and hence the dynamics is not measure preserving. The finite-dimensional subspace $\Fc_N$ is the space of all polynomials of degree no more than eight. For numerical stability we chose a basis of this subspace to be the Laguerre polynomials scaled such that they are orthonormal with respect to the uniform measure on $[-1,1]$. Spectra of $\Kc_N$ and $\Kc_{N,M}$ for $M=10^2$, $M=10^3$ and $M=10^5$ are depicted in Figure~\ref{fig:analDMDspec}. We observe that a relatively large number of samples $M$ is required to obtain an accurate approximation of the spectrum of $\Kc_{N}$. This example demonstrates that a special care must be taken in practice when drawing conclusions about spectral quantities based on a computation with a small number of samples. On the other hand, Figure~\ref{fig:analDMDpred} suggests that, at least on this example, predictions generated by $\Kc_{N,M}$ are less affected by sampling. Indeed, even for $M=100$ the prediction accuracy of $\Kc_{N,M}$ is comparable to that of $\Kc_N$ and for $M=1000$ the two predictions almost coincide.
\begin{figure*}[th]
\begin{picture}(140,135)
\put(-10,0){\includegraphics[width=60mm]{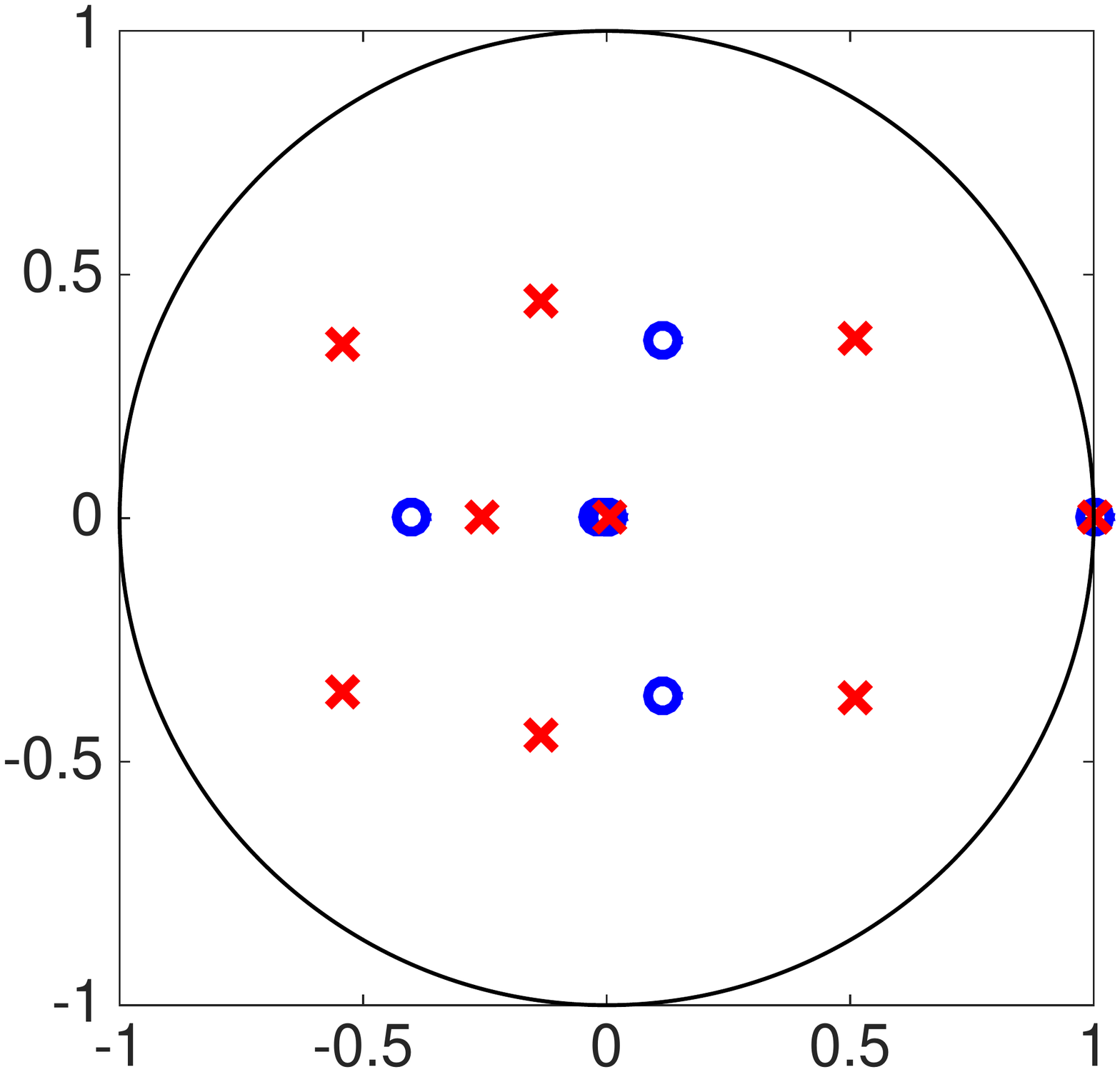}}
\put(155,0){\includegraphics[width=60mm]{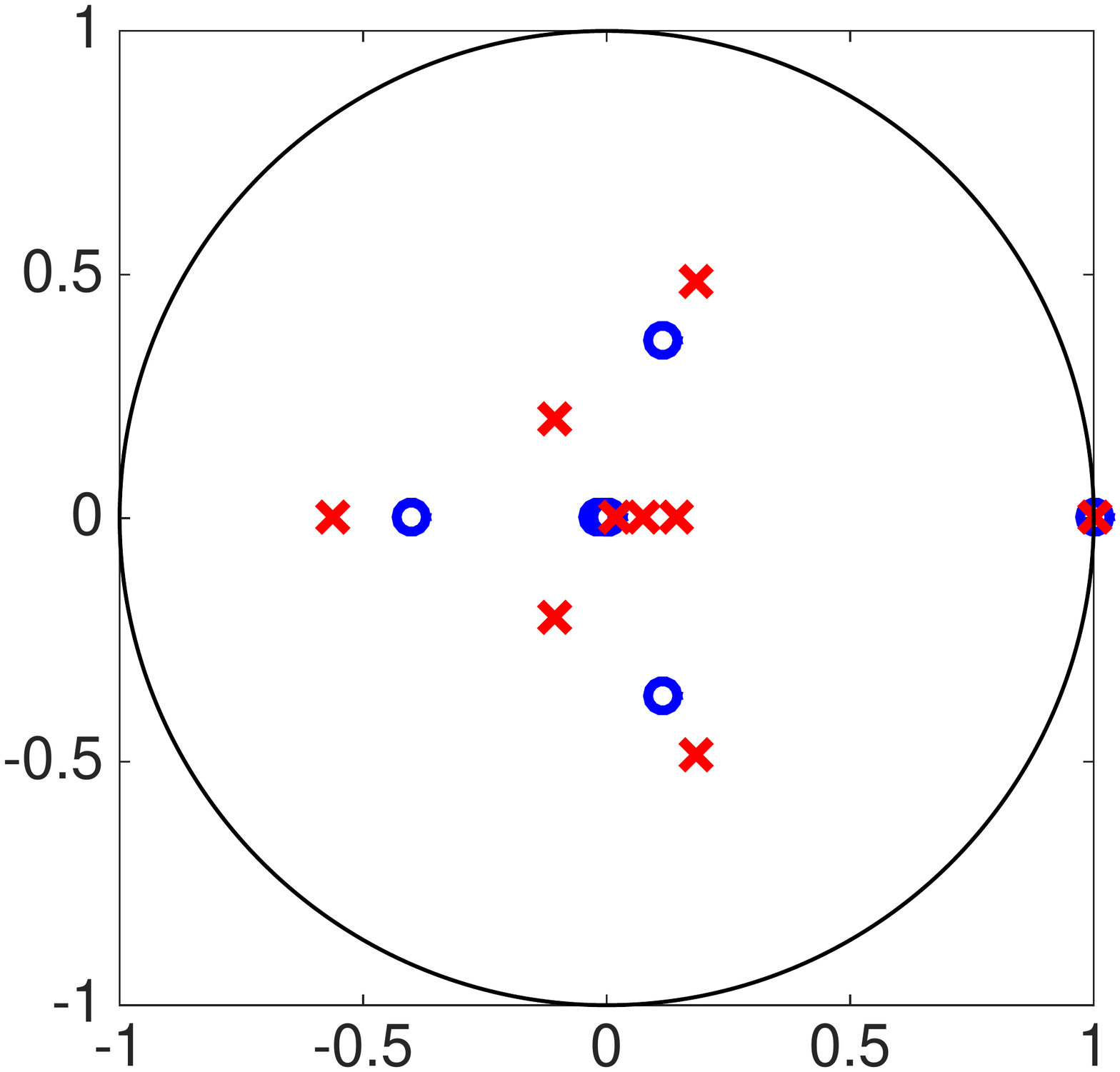}}
\put(320,0){\includegraphics[width=60mm]{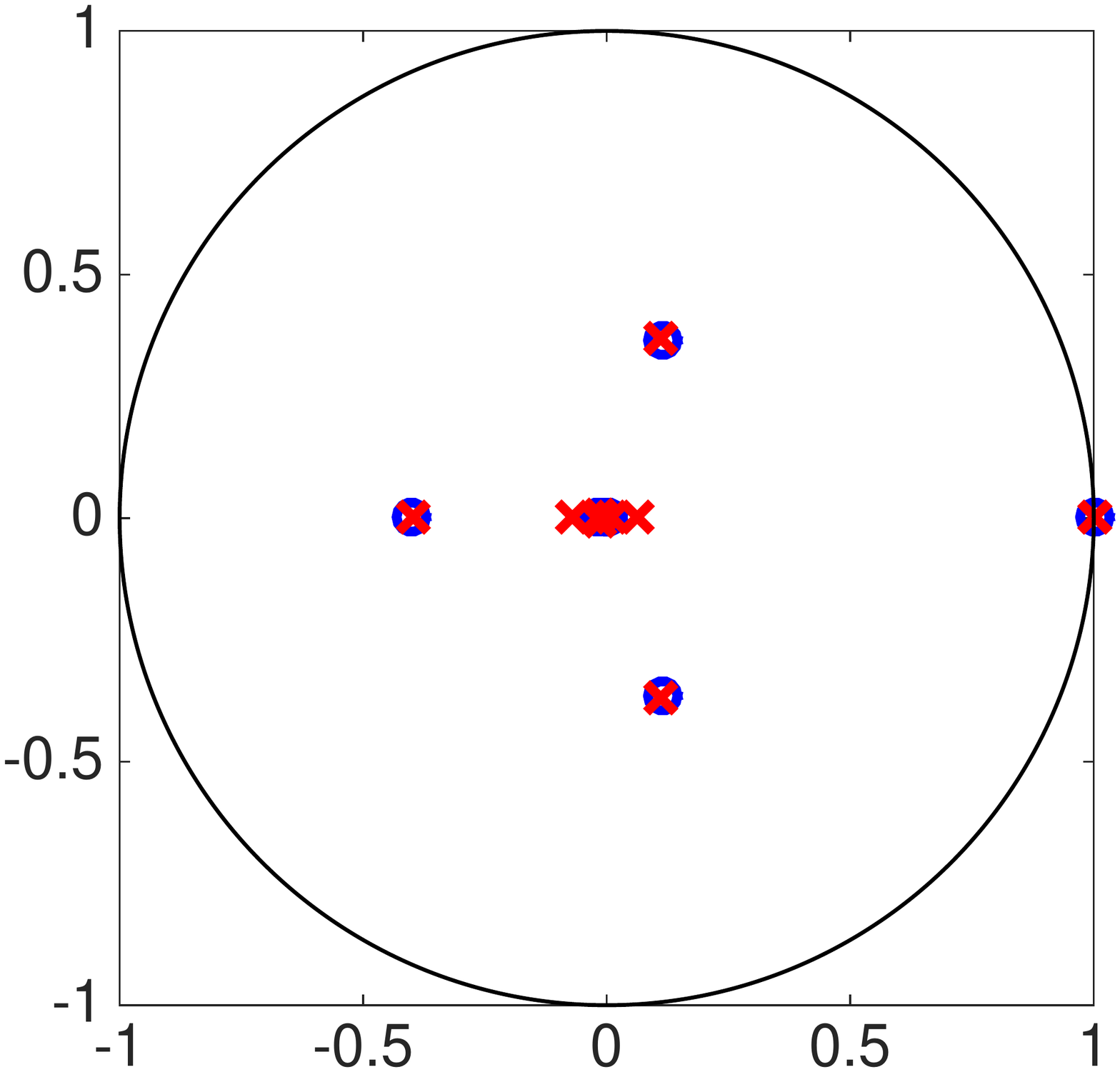}}

\put(230,0){\footnotesize $\Re\{\lambda\}$}
\put(-25,60){\footnotesize $\Im\{\lambda\}$}
\put(60,105){\footnotesize $ M = 10^2$}
\put(225,105){\footnotesize $ M = 10^3$}
\put(390,105){\footnotesize $ M = 10^5$}

\end{picture}
\caption{Comparison of the spectra of $\Kc_N$ (blue circles) computed using Analytic DMD and the spectra of $\Kc_{N,M}$ for different values of $M$ (red crosses).}
\label{fig:analDMDspec}
\end{figure*}

\begin{figure*}[th]
\begin{picture}(140,180)
\put(130,0){\includegraphics[width=80mm]{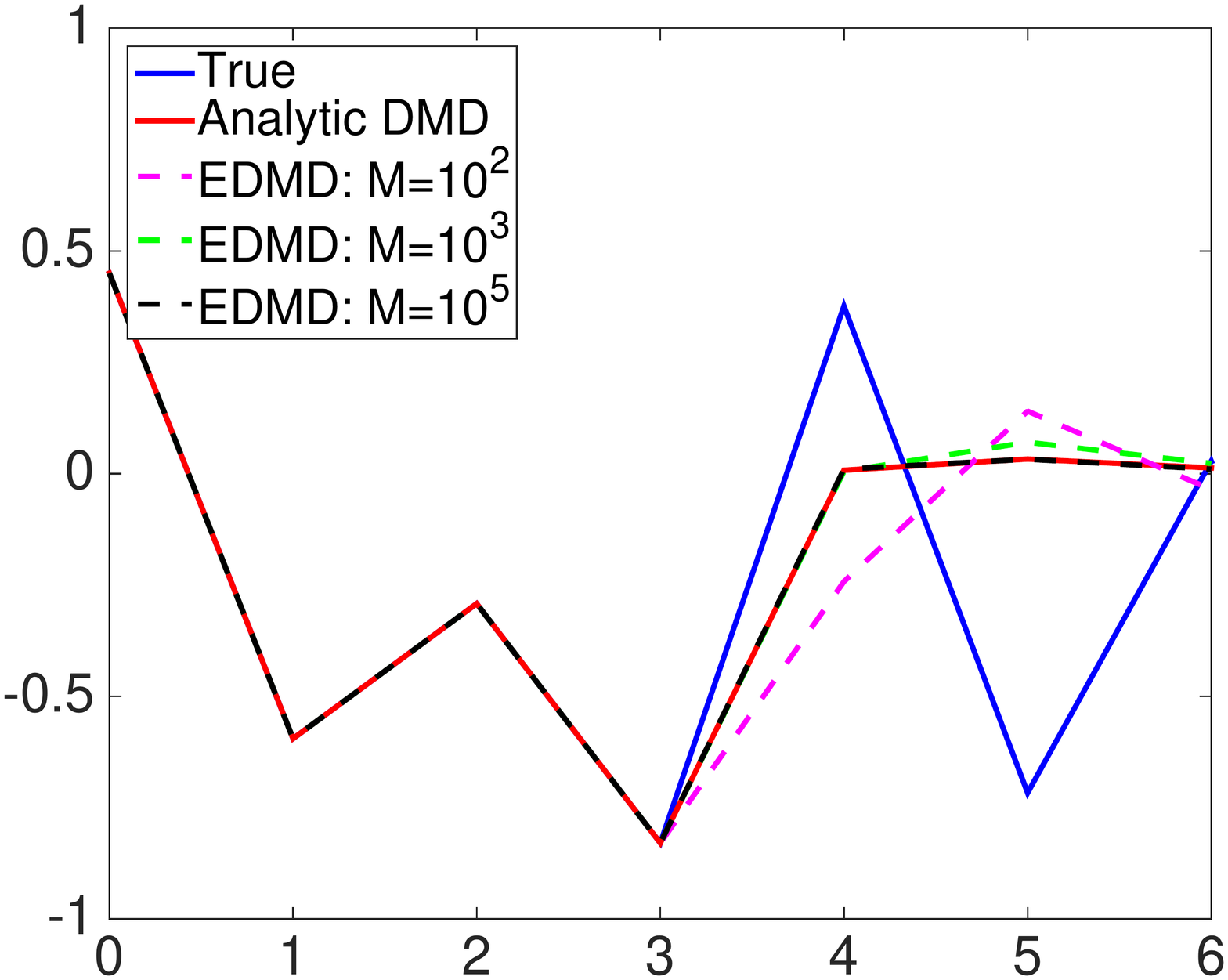}}
\put(235,0){time}
\put(120,90){ $x$}
\end{picture}
\caption{Comparison of predictions generated using $\Kc_N$ and using $\Kc_{N,M}$ for different values of $M$.}
\label{fig:analDMDpred}
\end{figure*}

}

\section{Convergence of $\Kc_{N,N}$}
\label{KNNconv}
In this section we investigate what happens when we simultaneously increase the number of basis functions $N$ and the number of samples $M$. We treat the special case of $M = N$ for which interesting conclusions can be drawn. Set therefore $M = N$ and denote $\lambda_N = \lambda_{N,N}$ any eigenvalue of $\Kc_{N,N}$ and $\phi_N = \phi_{N,N} \in \Fc_N$, $\|\phi_N\|_{C(\Mc)} =1$, the associated eigenfunction, where $\|\phi \|_{C(\Mc)} = \sup_{x\in\Mc} |\phi(x)|$; such normalization is possible if the basis functions $\psi_i$ are continuous and $\Mc$ compact, which we assume in this section. First notice that, assuming $M_{\hat\mu_N}$ invertible, for $N=M$ the system of equations
\[
\bs\psi(\bs Y) = A\bs\psi(\bs X)
\]
with the unknown $A\in \Rb^{N\times N}$ has a solution and hence the minimum in the least squares problem~(\ref{eq:ls}) is zero. In other words, for any $f\in\Fc_N$, the EDMD operator $\Kc_{N,N}$ applied to $f$ matches the value of the Koopman operator applied to $f$ on the samples points $x_1,\ldots,x_N$:
\[
(\Kc f)(x_i) = (\Kc_{N,N}f)(x_i)
\]
for all $f\in \Fc_N$. This relation is in particular satisfied for the eigenfunctions $\phi_N$ of $\Kc_{N,N}$, obtaining
\[
(\phi_N\circ T)(x_i) = \lambda_N \phi_N(x_i).
\]
Multiplying by an arbitrary $h\in\Fc$ and integrating with respect to the empirical measure supported on the sample points~(\ref{eq:emp_meas_def}), we get
\begin{equation}\label{eq:eig_sample}
\int_{\Mc} h\cdot (\phi_{N}\circ T)\, d\hat\mu_{N}  = \lambda_{N} \int_{\Mc} h \phi_{N} \, d\hat\mu_{N}.
\end{equation}
Define the linear functional $L_{N} : C(\Mc)\to \mathbb{C}$ by
\[
L_N(h) = \int_{\Mc} h\, \phi_{N}\, d\hat{\mu}_{N},
\]
and  
\[
(\Kc L_N)(h) = \int_{\Mc} h\cdot (\phi_{N}\circ T)\, d\hat\mu_{N}. 
\]
With this notation, the relationship~(\ref{eq:eig_sample}) becomes
\[
\Kc L_N = \lambda_{N} L_N
\]
Since $\| \phi_{N}\|_{C(\Mc)} = 1$, we have $\| L_N \| = \sup_{h\in C(\Mc)} \frac{|L_N(h)|}{\|h \|_{C(\Mc)}} \le 1$ and $\| \Kc L_N \| \le 1$. Therefore, assuming separability\footnote{A sufficient condition for $C(\Mc)$ to be separable is $\Mc$ compact and metrizable.} of $C(\Mc)$, by the Banach-Alaoglu theorem (e.g., \cite[Theorems 3.15, 3.17]{rudinFA}) there exists a subsequence, along which these functionals converge in the weak$^\star$ topology\footnote{A sequence of functionals $L_i\in C(\Mc)^\star$ converges in the weak$^\star$ topology if $\lim_{i\to\infty}L_i(f) = L(f)$ for all $f\in C(\Mc)$.} to some functionals $L \in C(\Mc)^\star$ and $\Kc L\in C(\Mc)^\star$ satisfying
\[
\Kc L = \lambda L,
\] 
where $\lambda$ is an accumulation point of $\lambda_N$. Furthermore by the Riesz representation theorem the bounded linear functionals $L$ and $\Kc L$ can be represented by complex-valued measures $\nu$ and $\Kc\nu$ on $\Mc$ satisfying
\[
\Kc\nu = \lambda\nu.
\]

We remark that $\Kc L$ and $\Kc\nu$ are here merely symbols for the weak$^\star$ limit of $\Kc L_N$ and its representation as a measure; in particular the functional $\Kc L$ is not necessarily of the form $(\Kc L)(h)=\int_{\Mc} h\cdot (\rho\circ T)\,d\mu$ for some function $\rho$.

In order to get more understanding of $\Kc\nu$ (and hence $\Kc L$), we need to impose additional assumptions on the structure of the problem. In particular we assume that the mapping $T:\Mc\to\Mc$ is a homeomorphism and that the points $x_1,\ldots,x_N$ lie on a single trajectory, i.e., $x_{i+1} = T(x_i)$. With this assumption, Equation~(\ref{eq:eig_sample}) reads
\begin{equation}\label{eq:eig_sample_1}
\frac{1}{N} \sum_{i=1}^N h(x_i)\phi_N(x_{i+1})  = \lambda_{N} \frac{1}{N} \sum_{i=1}^N h(x_i)\phi_N(x_i),
\end{equation}
where we set $x_{N+1}:= T(x_N)$. The left-hand side of~(\ref{eq:eig_sample_1}) is
\begin{align}
\frac{1}{N} \sum_{i=1}^N h(x_i)\phi_N(x_{i+1}) &= \frac{1}{N} \sum_{i=1}^N h(T^{-1} x_i)\phi_N(x_i) + \frac{1}{N}(h(x_N)\phi_N(x_{N+1}) - h(T^{-1}x_1)\phi_N(x_1))\nonumber \\
&= \int_{\Mc} h\circ T^{-1} \, d\nu_N + \frac{1}{N}(h(x_N)\phi_N(x_{N+1}) - h(T^{-1}x_1)\phi_N(x_1)), \label{eq:1Ncomp}
\end{align}
where $\nu_N$ is the measure $\phi_N d\hat\mu_N$. Setting $\xi_N:=h(x_N)\phi_N(x_{N+1}) - h(T^{-1}x_1)\phi_N(x_1)$, the relation~(\ref{eq:eig_sample_1}) becomes
\[
\int_{\Mc} h\circ T^{-1} \, d\nu_N + \frac{1}{N}\xi_N = \lambda_N\int_{\Mc} h\,d\nu_N.
\]
Since $h$ is bounded on $\Mc$ ($h$ is continuous and $\Mc$ compact) and $\| \phi_N\|_{C(\Mc)} = 1$, the term $\xi_N$ is bounded; in addition $h\circ T^{-1}$ is continuous since $T$ is a homeomorphism by assumption. Therefore, taking a limit on both sides, along a subsequence such that $\nu_{N_i}\to \nu$ weakly\footnote{A sequence of Borel measures $\mu_i$ converges weakly to a measure $\mu$ if $\lim_{i\to\infty}\int f\,d\mu_i = \int f\,d\mu $ for all continuous bounded functions $f$. This convergence is also referred to as narrow convergence and it coincides with convergence in the weak$^\star$ topology if the underlying space is compact (which is the case in our setting).}, $\hat\mu_{N_i}\to \mu$ weakly and $\lambda_{N_i}\to\lambda$, we obtain
\begin{equation}\label{eq:eigFunctional}
\int_{\Mc} h\circ T^{-1} \, d\nu = \lambda\int_{\Mc} h\,d\nu
\end{equation}
for all $h\in C(\Mc)$.

\subsection{Weak eigenfunctions / eigendistributions}
To understand relation~(\ref{eq:eigFunctional}), note that a completely analogous computation to~(\ref{eq:1Ncomp}) shows that the measure $\mu$ is invariant\footnote{A measure $
\mu$ on $\Mc$ is invariant if $\mu(T^{-1}(A)) = \mu(A)$ for all Borel sets $A\subset \Mc$ or equivalently if $\int_{\Mc} f\circ T \, d\mu = \int_{\Mc} f \, d\mu $ for all continuous bounded functions $f$.} and therefore the $L_2(\mu)$-adjoint $\Kc^\star$ of the Koopman operator (viewed as an operator from $L_2(\mu) $ to $L_2(\mu)$) is given by
\[
\Kc^\star f = f\circ T^{-1}.
\]
To see this, write
\[
 \langle \Kc f,g \rangle  = \int_{\Mc} (f\circ T) \overline{g} \,d\mu = \int_{\Mc} (f\circ T)\overline{(g\circ T^{-1}\circ T)} \,d\mu = \int_{\Mc} f \cdot \overline{(g\circ T^{-1} )}\,d\mu = \langle f, \Kc^\star g\rangle,
\]
which means the operator $g\mapsto g\circ T^{-1}$ is indeed the $L_2(\mu)$-adjoint of $\Kc$. The relation~(\ref{eq:eigFunctional}) then becomes
\begin{equation}\label{eq:eigFunctional_measurePres_1}
\int_{\Mc} \Kc^\star h \, d\nu = \lambda\int_{\Mc} h\,d\nu
\end{equation}
or
\begin{equation}\label{eq:eigFunctional_measurePres_2}
L(\Kc^\star h) = \lambda L(h).
\end{equation}

Functionals of the form~(\ref{eq:eigFunctional_measurePres_2}) were called ``generalized eigenfunctions'' by Gelfand and Shilov~ \cite{GelfandandShilov:1964}; here we prefer to call them ``weak eigenfunctions'' or ``eigendistributions'' in order to avoid confusion with generalized eigenfunctions viewed as an extension of the notion of generalized eigenvectors from linear algebra. The measure $\nu$ in~(\ref{eq:eigFunctional_measurePres_1}) is then called ``eigenmeasure''. Here, again, we emphasize the requirement that the limiting functional $L$ (or the measure~$\nu$) be nonzero in order for these objects to be called eigenfunctionals / eigenmeasures.

\subsection{Eigenmeasures of Perron-Frobenius}
We also observe an interesting connection to eigenmeasures of the Perron-Frobenius operator. To see this, set $h:= g\circ T$ in~(\ref{eq:eigFunctional}) to obtain
$
\int_{\Mc}g\,d\nu = \lambda\int_{\Mc} g\circ T\,d\nu
$
or, provided that $\lambda\ne 0$,
\begin{equation}
\int_{\Mc} g\circ T\,d\nu  = \frac{1}{\lambda} \int_{\Mc}g\,d\nu.
\end{equation}
In other words, if non-zero, the measure $\nu$ is the eigenmeasure of the Perron-Frobenius operator with eiegnvalue~$1/\lambda$. Here, the Perron-Frobenius operator $\Pc:M(\Mc)\to M(\Mc)$, where $M(\Mc)$ is the space of all complex-valued measures on $\Mc$, is defined for every $\eta \in M(\Mc)$ and every Borel set $A$ by
\[
(\Pc\eta)(A) = \eta(T^{-1}(A)).
\]

The results of Section~\ref{KNNconv} are summarized in the following theorem:
\begin{theorem}
Suppose that $\Mc$ is a compact metric space, $T$ is a homeomorphism, $\Kc:L_2(\mu)\to L_2(\mu)$ is bounded, the observables $\psi_1,\ldots,\psi_N$ are continuous and the sample points $x_1,\ldots, x_N$ satisfy $x_{i+1} = T(x_i)$ for all $i\in\{1,\ldots,N-1\}$. Let $\lambda_N$ be a bounded sequence of eigenvalues of $\Kc_{N,N}$,  let $\phi_N$, $\|\phi_N\|_{C(\Mc)}=1$, be the associated normalized eigenfunctions and denote $\nu_N = \phi_N d\hat \mu_N$. Then there exists a subsequence $(N_i)_{i=1}^\infty$ such that $\nu_{N_i}$ and $\hat\mu_{N_i}$ converge weakly to  complex-valued measures $\nu\in M(\Mc)$, $\mu\in M(\Mc)$ and $\lim_{i\to\infty} \lambda_{N_i } = \lambda \in \mathbb{C}$ such that
\[
\int_{\Mc} h\circ T^{-1} \, d\nu = \lambda\int_{\Mc} h\,d\nu\quad \forall\,h\in C(\Mc).
\]
In addition, the measure $\mu$ is invariant under the action of $T$ and
\[
\int_{\Mc} \Kc^\star h \, d\nu = \lambda\int_{\Mc} h\,d\nu \quad \forall\,h\in C(\Mc),
\]
where $\Kc^\star$ is the $L_2(\mu)$ adjoint of $\Kc$, i.e., if nonzero, $\nu$ is a weak eigenfunction (or eigendistribution) of the Koopman operator. Furthermore, if $\lambda \ne 0$, then
\[
\int_{\Mc} h\circ T\,d\nu  = \frac{1}{\lambda} \int_{\Mc}h\,d\nu \quad \forall\,h\in C(\Mc),
\]
i.e., if nonzero, $\nu$ is an eigenmeaure of the Perron-Frobenius operator with eigenvalue $1/\lambda$.
\end{theorem}

\section{Conclusions}
This paper analyzes the convergence of the EDMD operator $\Kc_{N,M}$, where $M$ is the number of samples and $N$ the number of observables used in EDMD. It was proven in~\cite{klus2015numerical} that as $M\to\infty$,  the operator $\Kc_{N,M}$ converges to $\Kc_N$, the orthogonal projection of the action of the Koopman operator on the span of the observables used in EDMD. We analyzed the convergence of $\Kc_N$ as $N\to\infty$, obtaining convergence in strong operator topology to the Koopman operator and weak convergence of the associated eigenfunctions along a subsequence together with the associated eigenvalues. In particular, any accumulation point of the spectra of $\Kc_N$ corresponding to a non-zero weak accumulation point of the eigenfunctions lies in the point spectrum of the Koopman operator $\Kc$. In addition we proved convergence of finite-horizon predictions obtained using $\Kc_N$ in the $L_2$ norm, a result important for practical applications such as forecasting, estimation and control. Finally we analyzed convergence of $\Kc_{N,N}$ (i.e., the situation where the number of samples and the number of basis functions is equal) under the assumptions that the sample points lie on the same trajectory. In this case one obtains convergence, along a subsequence, to a weak eigenfunction (or eigendistribution) of the Koopman operator, provided the weak limit is nonzero. This eigendistribution turns out to be also an eigenmeasure of the Perron-Frobenius operator. As a by-product of these results, we proposed an algorithm that, under some assumptions, allows one to construct  $\Kc_N$ directly, without the need for sampling, thereby eliminating the sampling error.

\new{Future work should focus on non-asymptotic analysis, e.g., on selecting the subspace $\Fc_N$ such that $\|\Kc_{N} - \Kc_{|\Fc_N} \|$ is minimized and at the same time such that $\Fc_N$ is rich enough in the sense of containing observables of practical interest (e.g., the state observable). This line of research was already investigated in the context of stochastic systems in~\cite{wu2017variational}, providing an interesting and actionable method for selecting $\Fc_N$. 
}

\label{conc}

\section{Acknowledgments}
The first author would like to thank Mihai Putinar for discussions on a related topic that sparked the work on this paper. We also thank Clancy Rowley for helpful comments on an earlier version of this manuscript as well as to P\' eter Koltai for point out the reference~\cite{klus2015numerical}. This research was supported in part by the ARO-MURI grant W911NF-14-1-0359 and the DARPA grant HR0011-16-C-0116. The research of M. Korda was supported by the Swiss National Science Foundation under grant P2ELP2\_165166.

\bibliographystyle{abbrv}
\bibliography{./References}

\end{document}